\documentclass[11pt]{article}

\usepackage{latexsym,amssymb,amsmath}

\newcommand{\Z}{\mathbb Z}
\newcommand{\Q}{\mathbb Q}

\newcommand{\F}{\mathbb F}

\newcommand{\pslp}{\mathrm{PSL}(2,p)}
\newcommand{\slp}{\mathrm{SL}(2,p)}

\newcommand{\slz}{\mathrm{SL}(2, \mathbb Z)}

\newcommand{\dih}[1]{\mathrm{D}_{#1}}

\newcommand{\sma}{\left(\begin{array}}
\newcommand{\fma}{\end{array}\right)}
\newcommand{\Alt}{\mathrm{A}}
\newcommand{\Sym}{\mathrm{S}}
\newcommand{\udot}{\mathord{{}^{\textstyle{{\cdot}}}{}}}
\newcommand{\lgl}{\langle}
\newcommand{\rgl}{\rangle}
\newcommand{\diag}{\mathrm{diag}}
\newcommand{\antidiag}{\mathrm{antidiag}}
\newcommand{\Magma}{{\sc Magma }}
\newcommand{\Magman}{{\sc Magma}}

\newtheorem{lem}{Lemma}[section]

\newtheorem{co}[lem]{Corollary}
\newtheorem{thm}[lem]{Theorem}
\newtheorem{prop}[lem]{Proposition}

\newenvironment{proof}{\textbf{Proof.}}{\hspace*{\fill}{$\Box$}\\}

\title{An explicit upper bound for the Helfgott delta in $\slp$}
\author{Jack Button and Colva Roney-Dougal}

\begin{document}

\maketitle

\begin{center}
{\em Dedicated to the memory of Akos Seress.}
\end{center}

\begin{abstract}
Helfgott proved that there exists a  $\delta>0$ such that if $S$ is a symmetric
generating subset of $\slp$ containing 1 then either $S^3=\slp$ or
$|S^3|\geq |S|^{1+\delta}$. It is known that $\delta\geq 1/3024$. Here
we show that $\delta\leq(\log_2(7)-1)/6 \approx 0.3012$
and we present evidence suggesting that this might be the true value
of $\delta$.
\end{abstract}
\section{Introduction}

A subset $S$ of a group $G$ is \emph{symmetric} if $S = S^{-1}$, that is if
$S$ is equal to $\{x^{-1}:x\in S\}$.
A very influential result \cite{hlf} of Helfgott 
(stated using the ``Gowers trick'' as in \cite[Corollary 2.6]{soda} 
is that 
there exists a $\delta>0$ such that if $S$ is a symmetric
generating subset of $G = \slp$ containing the identity
1 then the triple product 
$S^3$ is either equal to $G$ or has size at least $|S|^{1+\delta}$. This has immediate applications to the
diameter of Cayley graphs of $\slp$, and was also used by Bourgain and
Gamburd in \cite{bg} for the spectral gap of expander families of Cayley
graphs obtained from a Zariski-dense subgroup of  $\slz$ by reducing
modulo primes $p$. Recently, Helfgott and Seress generalised some of 
these ideas to prove a quasipolynomial bound on the diameter of the Cayley graphs 
of the alternating and symmetric groups \cite{HelfSer}. 

Helfgott's result can also be expressed in the
language of approximate groups, where a \emph{$k$-approximate group} $A$ is a
finite symmetric subset of a group $H$  such that $1 \in S$ and there exists
$X\subseteq H$ of size at most $k$ with $A^2\subseteq AX$. This immediately
implies that $|A^3|\leq k^2|A|$, so if $A$ is a generating $k$-approximate
group of $G=\slp$ then Helfgott's result tells us that either
$|A|\leq k^{2/\delta}$ or $|A|\geq |G|/k^2$.
Conversely, say there  exists an $N$ such that either $|A|\leq k^N$ or
$|A|\geq |G|/k^N$
for any generating $k$-approximate group $A$ of $G$. 
Then given $S$ a symmetric generating subset of $G$ containing 1, let
$k$ be such that $|S^3|=k^2|S|$. This implies (by Ruzsa's covering  lemma)
that $S^2$ is a $k^6$-approximate group. Here the Gowers trick tells us
that $S^3=G$ if $|S|\geq 2|G|^{8/9}$, so if the first case holds (namely
$|S|\leq |S^2|\leq k^{6N}$) we see that
$|S^3|/|S|=k^2\geq |S|^{1/3N}$. 
Now suppose that $|S^2|\geq |G|/k^{6N}$. If 
$k\leq 2^{-1/6N}|G|^{1/54N}$ then $S^6 = G$, and 
otherwise we can assume by the Gowers
trick again that $|S|<2|G|^{8/9}$, in which case 
$|S^3|/|S|>2^{-1/3N}|G|^{1/27N}$. Thus here $|S^3|>|S|^{1+\delta}$
provided that $2^{\delta}|G|^{8\delta/9}\leq 2^{-1/3N}|G|^{1/27N}$. This holds
for all but finitely many groups $G$ as long as we set
$\delta$ to be strictly less than $1/24N$, whereupon we can take
the minimum of this $\delta$ and
suitable values for the finitely many exceptions to obtain an overall
value of $\delta$ such that $|S^3| \geq |S|^{1+\delta}$ in all $G=\slp$. 
 
Not long after this, Helfgott's result was generalised to every family
of finite simple groups of Lie type with bounded Lie rank in \cite{ps},
with an equivalent version in \cite{bgt} expressed in terms of
approximate groups. 
Returning to $G = \slp$, in a recent paper \cite{kwl} by Kowalski the 
explicit lower
bound of $1/3024$ was shown to hold for $\delta$, by making
Helfgott's proof quantitative at every stage (this paper also
contains explicit versions of the two applications mentioned above).

Therefore define the
{\em Helfgott delta} in $G$ to be the supremum (which will be the
maximum) of the set
$\{\delta\in [0,\infty): |S^3|\geq |S|^{1+\delta}\}$ where $S$ ranges over
all symmetric generating sets of $\slp$ (over all primes $p$) that contain 1
and satisfy $S^3\neq G$. Given that this Helfgott $\delta$ must be at
least $1/3024$, one can also ask about a good upper bound, which is 
the topic of this paper. Establishing this has a different
flavour, because finding an explicit lower bound involves carefully inspecting
the whole of Helfgott's proof whereas we can be led by examples,
looking for such subsets $S$ where $\log(|S^3|)/\log(|S|)$ is as small as
possible. We shall take all logs to base $2$. 

The best upper bound we have found is
$(\log(7)-1)/6  \approx 0.3012$,  which comes from 
a symmetric subset $S$ containing 1 and generating $\slp$ that has size
64, whereas $|S^3|=224$. Moreover, such subsets can be found in $\slp$
for infinitely many primes $p$.

Our initial guess for subsets $S$ of small $\delta$
was that they should be as close to proper subgroups $H$ of $G$
as possible, so we started by looking at \emph{subgroup-plus-two
subsets}: these are sets of the form $H\cup\{x^{\pm 1}\}$ with 
$\langle H,x\rangle=\slp$. Note that as our subsets $S$ are symmetric,
 we need to add $x^{\pm 1}$ and not just $x$ to $H$. 
However it is a surprising result of this paper
that subgroup-plus-two subsets cannot be best possible as, regardless of
$H$ or $x$, they all produce a value of $\delta$ which is at least
$\log(3)/5 \approx 0.3169$.

We start by making some basic but useful observations in Sections 2 and 3.
In particular we show 
that for a subset $S=H\cup\{x^{\pm 1}\}$ in a group
$K$, the size of $S^3$ is controlled both above and below
by the index of $x^{-1}Hx\cap H$ in $H$. In addition, 
if $x^2\in H$ then $S^3 = H\cup HxH \cup x^{-1}Hx$, allowing us to obtain 
both tight upper and
lower bounds for $|S^3|$ in terms of $|H|$ and this index. In Theorem \ref{wlog} we 
show that, for general $x$,  if the expression for $S^3$ involves only one double coset $HxH$ then 
without loss of generality $x^2 \in H$. 

Then in Section 3 we display a construction that gives strictly
better results than subgroup-plus-two subsets. We call such a subset a \emph{subgroup
plus coset core} and they are introduced
after Proposition \ref{nwst}, where it is shown that if 
$S=H\cup\{x^{\pm 1}\}$, where $x^2\in H$, then there is an
obvious subset of $S^3$ that can be added to $S$ without adding new elements
to $S^3$. Moreover Proposition \ref{More} shows that this method cannot be
improved: given any symmetric subset $T$ containing a subgroup-plus-two subset $S = H \cup \{x, x^{-1}\}$
with $x^2 \in H$ and $T^3=S^3 \neq \slp$, the set 
$T$ is a subset of the subgroup plus coset core of $H$ and $x$. This provides further heuristic evidence that 
subgroup plus coset cores are likely to lead to small values of $\delta$. 

Consequently, for a given subgroup $H$ of $G = \slp$ we have a good strategy
for finding suitable sets with small triple product, by looking for an
element $x\in G\setminus H$ with $\langle H,x\rangle=G$ and $x^2\in H$ but
with $x^{-1}Hx\cap H$ having index
as small as possible in $H$, then taking the
subgroup plus coset core  associated to $H$ and $x$.
However, whilst minimising this index is a good proxy for obtaining a
small $\delta$ when $H$ is fixed, it is no good as $H$ varies because
subgroups of very large order could give rise, on choosing $x$,
to a high index but still do better in terms of $\delta$ than if a low
index was obtained from a
smaller subgroup. Fortunately the subgroup
structure of $\slp$ is very well known and we can therefore go through
all subgroups.

In Sections 4 and 5 we consider cyclic and dihedral subgroups, as well as
those conjugate into the subgroup of upper triangular matrices. We show
that for the latter subgroups $H$, as well as for cyclic groups $H$, any 
subgroup-plus-two subset
or subgroup plus coset core  $S$ formed from $H$
satisfies $|S^3|>|S|^{3/2}$, with a lower bound for the
dihedral subgroups.

Also in Section 5 we look at what might be termed the 
{\em eventual Helfgott delta}: one might only be interested in $\delta>0$
such that either $S^3=\slp$ or $|S^3|\geq |S|^{1+\delta}$ for sufficiently
large symmetric generating sets $S$ containing 1. In \cite{kwl} it was
mentioned that this $\delta$ is at least $1/1513$ and here we give an example
to show that it is at most $1/2$. 

In Section 6 we examine the exceptional subgroups 
$2 \udot \Alt_4$, $2 \udot \Sym_4$ and  $2 \udot \Alt_5$.
Basic estimates allow us to eliminate $2 \udot \Alt_4$ and $2 \udot \Alt_5$, then we
consider $2 \udot \Sym_4$ in more detail. Our best value of
$\delta$ is obtained by taking $H = 2 \udot \Sym_4$,  of order 48,  and
an element $x$ with $x^2\in H$ and such that $x^{-1}Hx\cap H$ has index
3 in $H$. We then let $S =  H \cup (xH\cap Hx)$, 
of size 64. We thus need to find the exact value of
$|S^3|$ and this is done in Theorem \ref{s4} by considering a particular
characteristic 0 representation of $H$. In Corollary 
\ref{inf} we show that this subset exists in $\slp$ for infinitely many primes 
$p$
and in Corollary \ref{best} show that it provides a strictly lower value of
$\delta$ than the infimum over
all other subgroup plus coset cores and all
subgroup-plus-two subsets, thus proving that the latter type of subset cannot
give rise to the minimal $\delta$.
 
It remains to be seen whether our subset provides the smallest value of
$\delta$ over all symmetric generating subsets $S$ with 1 where $S^3\neq\slp$,
as obviously we have attempted to guess the form of the best
subsets (and indeed our initial guess of subgroup-plus-two subsets was
not correct). However in Section 7 we provide further evidence as to why
our example $S$ might be best possible, in that it is robust with respect
to small perturbations and can be regarded as a local minimum. By this we
mean that if we remove an element and its inverse from $S$, or we add an
element and its inverse to $S$, or we do both operations simultaneously,
then the resulting subset produces a value for $\delta$ that is greater than $0.3012$.

Finally, we briefly discuss a complete search we did through $\mathrm{SL}(2, 5)$ 
using \Magma \cite{MAGMA}, and the optimal $\delta$ (which is around $0.3925$) 
and corresponding sets $S$. The sets $S$ which minimise $\delta$ for $p = 5$
are \emph{not} 
subgroup plus coset cores, but their structure is a little opaque to us -- we describe one such $S$. 
Since we submitted this paper, Christopher Jefferson has shown that all such sets $S$ are equivalent up
to conjugacy in $\mathrm{GL}(2, 5)$. 

\section{Background material}

Given a finite subset $S$ of a group $G$, we write $|S|$ for the size
of $S$. We also write $S^n$ for the $n$-th setwise product of $S$,
so for instance $S^3=\{abc: a\in S, b\in S, c\in S\}$.

Given subgroups $H$ and $L$ of a group $G$, for each $x\in G$
we can form
the double coset $HxL=\{hxl: h \in H, l\in L\}$.
We refer to \cite[Chapter II, Section 16 ]{led}
for the basic facts we will need. In
particular
\begin{prop}
\label{dbco}
(i) The group $G$ decomposes into a partition of double cosets
$Hx_iL$ for $i$ in some indexing set $I$.\\
(ii) (Frobenius) Let $d=|x^{-1}Hx\cap L|$. Then 
$$|HxL|=|H|\cdot |L|/d = |H| \cdot |L: x^{-1}Hx\cap L]. $$
\end{prop}

The following lemma is standard, see for example \cite[Satz II.8.27]{Huppert}.
\begin{lem}
\label{pslsubs}
Let $H$ be a subgroup of $\pslp$, $p \ge 5$. Then $H$ is one of:\\
(i) a subgroup of $C_p:C_{(p-1)/2}$, conjugate to the image of a group of upper triangular matrices;\\
(ii) a dihedral subgroup of the group $\dih{p-1}$ (of order $p-1$)\\
(iii) a subgroup of $\dih{p+1}$; \\
(iv) $\Sym_4$ (if and only if $p \equiv \pm 1 \bmod 8$) or $\Alt_{4}$;\\
(v) $\Alt_5$ (if and only if $p \equiv \pm 1 \bmod 10$).
\end{lem}

We will also use the following well known facts:
\begin{prop} 
\label{wk} Let $p \ge 5$. \\
(i) The only involution of $\slp$ is $-I$.\\
(ii) The only proper non trivial normal subgroup of
$\slp$ is $\{\pm I \}$.\\
(iii) Let
$\pi:\slp\rightarrow\pslp$ be the natural homomorphism
and $H$ be a subgroup of $\slp$. 
Then $-I\in H$ if and only if $H$ is even. Furthermore, $-I \in H$
if and only if the index $[\pslp:\pi(H)] = [\slp:H]$.
\end{prop}
\begin{proof}
A direct calculation, setting $A = A^{-1} \in \slp$ where $p \neq 2$, proves (i).
By \cite[Satz II.6.13]{Huppert} the group $\pslp$ is simple,
and we can pull back normal subgroups to get (ii). Part (iii) 
then follows from the fact that the index of a subgroup $H$
will be preserved under $\pi$ if and only if $H$
contains the kernel $\{\pm I\}$.
\end{proof}


\section{Potential subsets of small tripling}

Any proper subgroup $H$ of a finite group $G$ will be symmetric, contain the
identity $1$ and will satisfy $|H|=|H^3|$ $(=|H^n|)$ but of course will
not generate $G$. Moreover it is a straightforward exercise to show that
any subset $S$ of $G$ containing $1$ and with $|S|=|S^3|$ ($ = |S^2|$)
 is a subgroup
of $G$. Consequently our first candidates for symmetric generating sets $S$
which have small tripling and which contain $1$ are the subgroup-plus-two 
subsets $H \cup \{ x^{\pm 1}\}$, because they can generate $\slp$ but
we would expect that most of the growth in the size of $S^3$
would be absorbed by $H$.
Note that we are adding two distinct elements because if $|x| = 2$,
then $\langle H,x\rangle=H\times C_2 \neq \slp$  by Proposition \ref{wk}.

In this section we first show in Theorem \ref{31} that
our best subgroup-plus-two subsets $S=H \cup\{ x^{\pm 1}\}$ are likely to occur when
$x^2\in H$. However we then find in this case that we can obtain
an improved value of $\delta$ by adding elements to $S$ without
increasing the size of $S^3$, as shown in Propositions \ref{nwst} and
\ref{More}.

Let us now fix a subgroup $H$ and look for good heuristics to minimize 
$|S^3|$,  where $S = H \cup \{x^{\pm 1}\}$. 
We can express $S^3$ as the union of the thirteen subsets
\begin{equation}\label{eqn:H}
H,Hx^{\pm 1}H,x^{\pm 2}H, Hx^{\pm 2}, x^{\pm 1}Hx^{\pm 1},
x^{\pm 3}.
\end{equation}
Notice that if $x^2\in H$ then $S^3 = H \cup HxH \cup x^{-1}Hx$. 
It would seem that this gives
rise to the smallest tripling of $H$-plus-two subsets.
The following result show that if $S^3$ contains only two double cosets 
$H$ and $HxH$ then without loss of generality $x^2 \in H$. 
\begin{thm} \label{31}
\label{wlog}
Let $H \le G = \slp$ and $x \in G$ be such that $S=H\cup\{x^{\pm 1}\}$
satisfies $\langle S\rangle= G$.
Then either $HxH$ and $Hx^{-1}H$ are disjoint or there exists
$y\in Hx$ with  $y^2\in H$, such that
$T=H\cup\{y^{\pm 1}\}$ satisfies $\langle T\rangle=G$
and $|T|=|S|$ but $T^3\subseteq S^3$.
\end{thm}
\begin{proof}
Assume that
$HxH=Hx^{-1}H$. Thus $x=h_1x^{-1}h_2$ where $h_1,h_2\in H$, so on setting
$y=h_2^{-1}x$ we find that $y^2$ is equal to $h_2^{-1}h_1x^{-1}h_2$
times $h_2^{-1}x$ and so is in $H$. Consequently
$T^3$ is made up of the union of $H,HyH$ and $y^{-1}Hy$ which are
equal to $H,HxH$ and $x^{-1}Hx$ respectively, thus $T^3\subseteq S^3$.
Moreover $\langle H,x\rangle=\langle H,y\rangle=G$ and so 
$y\neq y^{-1}$, giving $|T|=|S|$.
\end{proof}

However, it could be that there are elements $y\in S^3$ with the
property that $(S\cup\{y^{\pm 1}\})^3=S^3$, thus increasing $|S|$ but
keeping $|S^3|$ constant to obtain a smaller $\delta$. In the case where
$x^2\in H$ quite a few such elements can be added  in this
way. From now on, given  a subgroup-plus-two subset $H\cup\{x^{\pm 1}\}$,
we let $L$ be the intersection $H\cap x^{-1}Hx$.

\begin{prop} \label{nwst}
Let $H$ be a proper subgroup of the finite group $K$, let $S=H\cup\{x^{\pm 1}\}$ with
$x^2\in H$, and set $T=H\cup xL$. Then $|T|\geq |S|$ but
$T^3=S^3$.
\end{prop}
\begin{proof}
Now,  $x^{-1}Lx=x^{-1}Hx\cap x^{-2}Hx^2=L$ so $xL = Lx$. We
look at the subsets listed in Equation~\ref{eqn:H}, 
but with $xL = Lx$ in place of $x$, and notice that 
the expressions simplify to give $T^3 = H \cup HxH \cup x^{-1}Hx$. 
\end{proof}

Note that $xL=xH\cap Hx$ and that $x^{-1}\in xL$ if and only if $x^2\in H$,
so $x^2\notin H$ implies that $H\cup xL$ is not a symmetric subset.
Moreover, if $x^2\in L$ then $g^2\in L$ for all $g\in xL$.
Consequently, if $x^2\in H$ then we will call
$H\cup (xH\cap Hx)$ a
{\it subgroup plus coset core}.
We now check that there are no further elements that can be
added to a subgroup-plus-two subset $S$ in a group $K$ without increasing
the size of $S^3$, assuming that $S^3 \neq K$. 

\begin{prop} \label{More}
Let $K$ be a finite group, let $H$ be a non-normal subgroup of $K$, let $x \in K$ such that 
$\langle H, x \rangle = K$ and $x^2 \in H$ with $|x| > 2$,  and define $S = \{H, x^{\pm 1}\}$. 
If $S^3 \neq K$, 
then the largest subset $T$ of $K$ satisfying  $S^3= T^3$ with $T = T^{-1}$ 
and $S \subset T$ is $T = H \cup (Hx \cap xH)$. 
\end{prop}

\begin{proof}
Let $y \in T \setminus H$. We shall show that $y \in Hx \cap xH$.

Our assumption that $x^2 \in H$ implies that $S^3 = H \cup HxH \cup x^{-1}Hx$. 
Now,  $T^3=S^3$ implies that $HyH \subset S^3$, and $HyH$ is an 
$(H, H)$-double coset that is not equal to $H$. If $HyH \neq HxH$ then 
$HyH$ has trivial intersection with both $H$ and $HxH$, so
$HyH \subseteq S^3$ implies that $HyH \subset x^{-1}Hx$, a 
contradiction since $|HyH| \geq |H|$ and $1 \not\in HyH$. 
So $HyH = HxH$, and in particular, $\langle H, y\rangle = \langle H, x \rangle = K$. 

Let the right coset representatives of $H$ in $HxH$ be 
$1 = t_0, x = t_1, xh_2 = t_2, \ldots, t_k$. If
$S^3 \neq K$, then there are right cosets of $H$ in $K$ that do not lie in 
$H \cup HxH$. 

Consider the action of $K$ on the right cosets of $H$, and identify the coset $Ht_i$ with $i$. 
Then $\{0\}$ and $\{1, \ldots, k\}$ 
are $H$-orbits in this action, and $0^y \in \{1, \ldots, k\}$,  
so $y$ must map at least one element of 
$\{1, \ldots, k\}$ outside of $\{0, \ldots, k\}$ because $y$ and $H$
generate $K$.
That is, there exists an $i \in \{1, \ldots, k\}$ such that 
$t_i y = xh_iy \not\in H \cup HxH$. Now, 
$t_i y = xh_i y \in S^3$ implies that $x h_i y \in xHx$, and so $y \in Hx$.

Similarly, let the left coset representatives of $H$ in $HxH$ be 
$s_1 = x, s_2 = h_2^\prime x,  \ldots, s_k = h_k'x$.
The group $K$ also acts on the set of all right $H$-cosets, via $(s_iH)^g = g^{-1}s_iH$,
and there exists an 
$i \in \{1, \ldots, k\}$ such that 
$(s_iH)^{y^{-1}} = ys_i H \not\in H \cup HxH$. If $ys_i \in S^3$ 
then $yh'_ix \in xHx$ so $y \in xH$.
\end{proof}


We now present two results which we will use to calculate or bound values of $\delta$ for 
various explicit subsets $S$. First, in Proposition~\ref{ind12} we collect information about what can happen 
when $HxH$ is a union of few $H$-cosets. 

\begin{prop} \label{ind12}
Let $H$ be a proper subgroup of the finite group $K$, with $\langle H,x\rangle=K$.\\
(i) If $|HxH|=|H|$ then $H$ is normal in $K$, thus
$K \neq \slp$.\\
(ii) If $|HxH|=2|H|$ and $HxH = Hx^{-1}H$ then
$L=H\cap x^{-1}Hx$ is normal in $K$, thus again
$K \neq \slp$. 
\end{prop}
\begin{proof} The first condition
implies that $x^{-1}Hx=H$ by Proposition \ref{dbco} (ii).
Thus $H$ is normalised by $\lgl H, x \rangle = K$. If $K=\slp$
then $H=\{I\}$ or $\{\pm I\}$ 
by Proposition~\ref{wk}\,(ii).
But  then $H\cup\{x\}$ will not generate $\slp$. 

As for (ii), if $|HxH|=2|H|$ then $[L:H] = 2$,  so $L \unlhd H$.
In addition,  $HxH=Hx^{-1}H$, so if $x^2\notin H$ then by Theorem \ref{wlog} 
we can change $x$ if necessary, but 
keeping the same $H,HxH$ and $x^{-1}Hx$, and thus the same $L$.
As the new and old $x$ are in the same right coset of $H$, we still have
$\langle H,x\rangle=K$ but
$x^{-1}Lx=x^{-1}Hx\cap H=L$ as now
$x^2\in H$, thus $L \unlhd K$. 

If $K = \slp$ then $L \le \langle -I \rangle$. If $L=\{I\}$ then we have the same contradiction as above,
whereas if $L=\{\pm I\}$ then let $\overline{H}$ and $\overline{x}$ be their images
in $\pslp$. Now $\overline{H} \cong C_2$ and $x^2\in H$, so either
$\overline{x}^2$ is the identity in $\pslp$ so that $\langle \overline{H},\overline{x}\rangle$ is a dihedral
group, or $\overline{x}^2$ generates $\overline{H}$ and $\langle \overline{H},\overline{x}\rangle$ is cyclic.
Either way $\langle \overline{H},\overline{x}\rangle\neq\pslp$ so 
$\langle H,x\rangle\neq\slp$.
\end{proof}

Since $x\notin H$, the sets
$H$ and $HxH$ are disjoint. 
Let $c = [H:H\cap x^{-1}Hx]$, and set $S = H \cup \{x^{\pm 1}\}$. 
Then from Proposition \ref{dbco}\,(ii), 
we deduce that
$|HxH|+|H|=(c+1)|H|\leq |S^3|$. Moreover, by
Theorem \ref{wlog}, without loss of generality
either $x^2\in H$, in which case $S^3 = H \cup HxH \cup x^{-1}Hx$, 
and so $|S^3|\leq (c+2-1/c)|H|$, or $x^2 \not\in H$, in which case
$HxH\cup Hx^{-1}H\cup H$ is a disjoint union, and 
$|HxH|+|Hx^{-1}H|+|H|=(2c+1)|H|\leq |S^3|$.

The following technical result, which follows from the preceding paragraph,
 will be used repeatedly to show that $\delta = 
(\log_2(7) - 1)/6$ is minimal over all subgroup-plus-two subsets and subgroup plus coset cores.

\begin{lem}\label{lem:c_estimate}
Let $H$ be a non-normal subgroup of a finite group $K$, let $x \in K$ be such that 
$\langle H, x \rangle = K$ and $|x| > 2$, 
let $L = H \cap x^{-1}Hx$ and $c = [H:L]$. If $HxH \neq Hx^{-1}H$ then let $S = H \cup 
\{x, x^{-1}\}$; otherwise assume that $x^2 \in H$ and let $S = H \cup xL$. \\
(i) If $HxH \neq Hx^{-1}H$ (which will hold when $c=2$ by Proposition \ref{ind12}~(ii) 
if $K=\slp$) then $|S^3| \geq (2c+1)H$.\\
(ii) Otherwise,  $(c+2-1/c)|H|\geq |S^3|\geq (c+1)|H|$ and $|S| = (1 + 1/c)|H|$.
\end{lem}

However it is less clear
how to proceed once $|H|$ varies. For instance,
given $H \le \slp$ with $|H|=12$ and $x$ as in Lemma~\ref{lem:c_estimate}\,(ii) with
$c=3$, the set $S=H\cup xL$ has size 16 and $48\leq|S^3|\leq 56$, giving
a value for $\delta$ of between $\log(48)/4-1 \approx 0.3962$
and $\log(56)/4-1 \approx 0.4518$ which we
might think is nice and low. However, given another subgroup
$K$ of order 144 and $z$ with $z^2 \in K$ where the index 
$[K: z^{-1}Kz\cap K]$ is as much as 6, we find that $|S|=168$ and
$|S^3|\leq (8-1/6)\cdot 144=1128$, giving $\delta \leq \log(1128)/\log(168)-1
\approx 0.3716$ which beats the lower estimate above.

However, the subgroups of $\slp$ are well studied, so in the next two
sections we shall look at the infinite families of subgroups in $\slp$, 
where we
are able to get stronger lower bounds on $\delta$ for subgroup-plus-two
subsets
and subgroup plus coset cores than would be implied by
the estimates above. We then look in Section 6
at the exceptional subgroups and their
small index subgroups, which is where our lowest value of $\delta$ shall
be obtained.

We finish this section with two useful inequalities which will come into
play when we consider specific subgroups of $\slp$.
\begin{lem} \label{log}
If $k\geq 1$ and $l\geq 2$
then $f_l(k)=\log(lk(k+1))/\log(l(k+1))$
and $g_l(k)=\log(lk(2k+1))/\log(l(k+1))$ are both increasing in $k$.
\end{lem}
\begin{proof}
We can write $f(k)=1+\log(k)/\log(l(k+1))$ then take derivatives and
rearrange to find that $f'(k)>0$.
We then do the same for 
$g(k)=\log(lk)/\log(l(k+1))+\log(2k+1)/\log(l(k+1))$.
\end{proof}

\section{Cyclic and Dihedral subgroups}

We start with a general lemma which comes in useful for cyclic
groups.
\begin{lem} \label{unin}
Suppose that $H$ is a proper subgroup of a finite group $K$
and that $L=x^{-1}Hx\cap H$ for some $x\in K$. If $L$ is the only
subgroup of $H$ with that index then $L$ is normalised by $x$.
\end{lem}
\begin{proof}
If $L$ has order $l$ and is the only subgroup of index $i$ in $H$ then
$x^{-1}Lx$ is the only subgroup of index $i$ in the order $li$
group $x^{-1}Hx$. But $L$ is also an order $l$ subgroup of  
$x^{-1}Hx$, thus it is of index $i$ and so $L=x^{-1}Lx$.
\end{proof}

Let us now consider the case where  $H=\langle z\rangle$, and 
$S= H \cup \{x^{\pm 1}\}$ or $S = H \cup (xH \cup Hx)$. 
We can certainly find $x\in G = \slp$ with $\langle H\cup\{x\} \rangle = G$, because 
$G$ is $2$-generated for all $p$. However we
will now see that the possibilities for $|S^3|$ are limited. 

\begin{prop} \label{cyc}
Let $H = \langle z \rangle \le G = \slp$, and let
$S=H\cup\{x^{\pm 1}\}$,  or let $x^2\in H$ and $S=H\cup (xH\cap Hx)$. 
If $\langle S \rangle=\slp$ then $|S^3|\geq |S|^{1+\delta}$, where
$\delta=\log(3)/3 \approx 0.5283$. 
\end{prop}
\begin{proof}
Set $L=x^{-1}Hx\cap H$, then  $L \unlhd 
H$, and Lemma~\ref{unin} implies
that $x^{-1}Lx = L$. This forces $L$ to be a proper normal
subgroup of $G$, so $L \leq \{\pm I\}$ by Proposition~\ref{wk}, and setting $n = |H|$ we see that
$[H:L] \ge  n/2$. 

First suppose
that $HxH = Hx^{-1}H$. By Theorem~\ref{wlog}
there exists $y\in Hx$ such that $y^2\in H$, but then $y^2\in
y^{-1}Hy=x^{-1}Hx$, thus $y^2\in x^{-1}Lx=L$. If $L = I$ then $y=-I$, but then
$\langle H,y\rangle=\langle H,x\rangle \neq G$, a contradiction.
Thus $L=\{\pm I\}$ and $yL=\{y^{\pm 1}\}$
so we can regard subgroup-plus-two subsets and subgroup plus coset
cores as equal, and $|S| = n+2$. 
Then Lemma~\ref{lem:c_estimate}\,(ii) bounds $|S^3|\geq (n/2 + 1)n,$
where $n$ is even and at least 4.
But $y^2=-I$ so that if $n=4$ then the
image of $\langle H,y\rangle$ in $\pslp$ is dihedral.
So $n \ge 6$ and we
are done if $(n/2+1)n\geq (n+2)^{1+\delta}$, which by taking logs and setting
$l=2$ and $k=n/2$ is equivalent to claiming that $f_2(k)\geq 1+\delta$.
But as $k\geq 3$ we get $f_2(k)\geq f_2(3)=1+\log(3)/3$ by
Lemma~\ref{log}, so this value of $\delta$ works.

Next suppose that $HxH \cap Hx^{-1}H = \emptyset$, so that  $|S| = n+2$.
Then  Lemma~\ref{lem:c_estimate}\,(i) bounds 
$|S^3| \geq (n+1)n.$ Thus we can again set $l=2$ and $k=n/2$ for
$k\geq 3/2$ (as $n\geq 3$)
in Lemma~\ref{log} for $g_2(k)$, meaning that we require
$g_2(k)\geq 1+\delta$. But we know $g_2(k)\geq g_2(3/2)=1+\log(12/5)/\log(5)
> 1+\log(3)/3$.
\end{proof} 

We can now move on to the dihedral subgroups arising in 
Proposition~\ref{wk}, so that $-I \in H$.
Indeed if the image in $\pslp$ is the dihedral group $\dih{2n}$ of order $2n$ then
$H$ has the presentation
\[\langle z,w| z^{2n}, w^4, z^n=w^2, w^{-1}zw=z^{-1}\rangle\] 
with $w^2$ being equal to $-I$,
which is known as the generalized quaternion group $Q_{4n}$.
We can mostly proceed by reducing to the cyclic case, although
the estimates obtained for $\delta$ will necessarily be lower.

\begin{prop} \label{dih}
Let $H=\langle z,w\rangle \cong 2 \udot \dih{2n}$ be a subgroup of
$G = \slp$, and let 
 $S=H\cup\{x^{\pm 1}\}$, or let $x^2\in H$ and $S=H\cup (xH\cap Hx)$.
If $\langle S \rangle=G$, then $|S^3|\geq |S|^{1+\delta}$ where
$\delta=\log(3)/5  \approx 0.3169$.
\end{prop}
\begin{proof} 
The group $C=\langle z\rangle$
of order $2n$ has index 2 in $H$, so in analogy with the proof above
we set $M=x^{-1}Cx\cap C$ and obtain in the same way that $x^{-1}Mx = M$. However any subgroup of $C$ is normalised by $H$, so once again we conclude that $M=\{I\}$ or $\{\pm I\}$. But
$-I \in C$, so $M = \{\pm I\}$.

Now if $A,B,D$ are subgroups of $G$ and $A$ is contained in $B$ with
index $i$ then $A\cap D$ has index at most $i$ in $B\cap D$. 
As  $[H : C] = 2$, and $[x^{-1}Hx : x^{-1}Cx] = 2$ also, 
the group $M$ has index at most 2 in
$x^{-1}Hx\cap C$, which has index at most 2 in $L=x^{-1}Hx\cap H$, thus
$|L|$ is $2, 4$, or $8$. Let $c = [H:L]$.

First suppose that $HxH = Hx^{-1}H$, so by
Theorem~\ref{wlog} there exists $y\in Hx$ with $y^2\in L$, and $c \ge 3$ by Proposition~\ref{ind12}.
By Lemma~\ref{lem:c_estimate}\,(ii), the set $S$ has size at most $(c+1)|L|$ whereas
$|S^3|\geq (c+1)|H|=c(c+1)|L|$. We can apply Lemma~\ref{log}
for $l=|L|=2,4,8$ by taking $k=c=2n,n$ and $n/2$,
respectively,  giving 
$f_2(k)\geq f_2(4)$, $f_4(k)\geq f_4(3)$ and $f_8(k)\geq f_8(3)$.
Of these the lowest value is $f_8(3)=\log(96)/5=1+\log(3)/5
\approx 1.3169$. 

Finally if $HxH$ and $Hx^{-1}H$ are disjoint then 
Lemma~\ref{lem:c_estimate}\,(i) gives $|S^3|\geq c(2c+1)|L|$ so we again set
$l=|L|=2,4,8$ and $k=c=2n,n$ and $n/2$
to obtain
$g_2(k)\geq g_2(2)$, $g_4(k)\geq g_4(2)$ and $g_8(k)\geq g_8(2)$, all
of which lie comfortably above $1+\delta$.
\end{proof}

\section{Triangular subgroups}\label{sec:tri}

The group $\slp$ has a subgroup
\[U=\{\sma{cl} \alpha&\beta\\ 0&\alpha^{-1}\fma: \alpha\in\Z_p^*,
\beta\in\Z_p\}\]
which is maximal and has order $p(p-1)$. In this section we will
assume that $H$ is any subgroup of $U$ and that $x\notin U$. This assumption is valid because any other subgroup of $\slp$ of 
order dividing $p(p-1)$ is conjugate to a subgroup of $U$, and the size of triple products is preserved by conjugation. 

In this and the next section we will need some additional notation for matrices in $\slp$. We write 
$u(\alpha, \beta)$ for $\sma{cl}\alpha&\beta\\0&\alpha^{-1}\fma\in U$, write $\diag[\alpha, \beta]$ for the diagonal matrix with entries $\alpha, \beta$, and write $\antidiag[\alpha, \beta]$ for the antidiagonal matrix with $\alpha$ in row $1$. 

\begin{thm} \label{uptr}
Let $H$ be a subgroup of $U$. 
If $S=H\cup\{x^{\pm 1}\}$, or $x^2\in H$ and $S=H\cup (xH\cap Hx)$, and 
$\langle S \rangle=\slp$, then $|S^3|>|S|^{3/2}$.
\end{thm}
\begin{proof}
First note that $U$ splits as the semidirect product $N\rtimes D$ where
\[N=\{u(1, b) 
:b\in \Z_p\}\mbox{ and }
D=\{ \diag[\lambda, \lambda^{-1}]
:\lambda\in\Z_p^*\}.\]
Since $N$ is simple, either
$H\cap N=\{I\}$ in which case $H$ is cyclic and the result follows from Proposition~\ref{cyc}, 
or
$N\leq H$, which we assume from now on. We let $x = \sma{cc} a&b\\c&d\fma \in \slp$
and
count the set
\[\{h \in H :xhx^{-1}\in H\}= \{h \in H:
xhx^{-1}=\sma{cc} *&*\\0&*\fma\}.\]
This equality is because if 
$xu(\alpha, \beta) x^{-1} = 
u(\gamma, \delta)$
then the traces are the same,
giving $\alpha=\gamma^{\pm 1}$. But if
$u(\alpha, \beta) \in H$ then so is
$u(\alpha^{\pm 1}, \eta)$
for any $\eta\in\Z_p$ because $N\leq H$.

The $(2, 1)$-entry of
$xu(\alpha, \beta)x^{-1}$ is
$(\alpha-\alpha^{-1})dc-\beta c^2$. 
As $c\neq 0$, this is zero if and
only if $(\alpha-\alpha^{-1})dc^{-1}=\beta$. Thus, as $x$ is fixed,
for each $\alpha\in\Z_p^*$ such that $u(\alpha, \beta) \in H$ for at least one $\beta$, only one such $\beta$ satisfies
$u(\alpha, \beta) \in H \cap x^{-1}
 H x$. Therefore, 
$|H\cap x^{-1}Hx|=
|H|/p$ and thus $|HxH|=|H|^2/|H\cap x^{-1}Hx|=p|H|$.
Thus by Lemma~\ref{lem:c_estimate}\,(ii), 
$|S^3|\geq (p+1)|H|$
and $|S|\leq (1+1/p)|H|$. Now $p$ divides $|H|$ so set $|H|=pk$.
Thus we require $(p+1)pk>k^{3/2}(p+1)^{3/2}$. By rearranging
and squaring we obtain $p^2/(p+1)>k$. Now $|H|\leq p(p-1)$ so $k \leq p-1$ and we are done. 
\end{proof}  

A variation on the Helfgott result for $\slp$ is that there exist two 
absolute constants $c,\delta>0$ such that for any symmetric generating
subset $S$ containing 1, either $S^3=\slp$ or $|S^3|\geq c|S|^{1+\delta}$.
To relate this to our formulation, this variation essentially says
that $|S^3|\geq |S|^{1+\delta}$ for all sufficiently large $|S|$. Indeed,
if the latter holds for all such $S$ with $|S|\geq N$, set 
$c=N^{-\delta}$ and keep the same $\delta$. If however 
$|S^3|\geq c|S|^{1+\delta}$ then although this need not ensure that
$|S^3|\geq |S|^{1+\delta}$ for all large $|S|$, we will have
$|S^3|>|S|^{1+\delta'}$ for any $\delta'<\delta$. Therefore we can
introduce the following notion: let $\Delta$ be the set of real positive
numbers $r$ such that $|S^3|\geq |S|^{1+r}$ for all sufficiently large
symmetric generating subsets
$S$ of $\slp$ containing 1 and with $S^3\neq\slp$. We define the
{\em eventual Helfgott delta} to be the supremum of $\Delta$. 
The next pair of 
results show  that this $\delta$ must be at most $1/2$.

\begin{prop} \label{evdlt}
If $p$ is a prime congruent to $1 \bmod 4$ then there is a symmetric subset
$S$ of $\slp$ containing 1 of size $\frac{p(p-1)+4}{2}$ such that
$(p+1)p(p-1)/2\leq |S^3| \leq (p+2)p(p-1)/2$.
\end{prop}
\begin{proof} One might first try applying Theorem~\ref{uptr} 
to the subgroup-plus-two subset $S=H\cup\{x^{\pm 1}\}$ 
with $H$ the subgroup $U$
of upper triangular matrices and $x\in \slp$ chosen so that $x^2\in H$ and
$\langle x,H\rangle=\slp$. The problem is that we find from the proof that
$|S^3|\geq (p+1)p(p-1)$ which is all of $\slp$. Consequently we set $Q$ to be
the set of quadratic residues mod $p$, with $\pm 1\in Q$ and we let $H$ be the
index 2 subgroup of $U$
$$\{u(q, \beta) : q\in Q, \beta\in \Z_p\}$$
of order $p(p-1)/2$. Now we find a suitable $x$, for instance $x$ could be
the order 4 element $\sma{rr}1&-2\\-1&-1\fma$ with $x\notin U$ but 
$x^2=-I\in H$. Then Theorem~\ref{uptr} gives us that
\[|S^3|\geq |HxH|+|H|=(p+1)|H|.\]
But as $x^2\in H$, we can use the argument just before Theorem~\ref{wlog}
to say that
$|S^3|\leq |HxH|+|H|+|xHx^{-1}|=(p+2)|H|$.
\end{proof}

\begin{co} \label{evhlf}
The eventual Helfgott delta is at most $1/2$.
\end{co}
\begin{proof}
On taking $S$ as in Proposition~\ref{evdlt} we see that
$|\slp|/2\leq|S^3|\leq (p+2)p(p-1)/2
<|\slp|=(p+1)p(p-1)$, thus $S^3\neq\slp$ and
as $p$ tends to infinity, $|S^3|/|S|^{3/2}$ tends to $2^{1/2}$ by squeezing. 
Now if $S$ generated a proper subgroup
of $\slp$ then this subgroup would have index 2 and so be normal, which
contradicts Proposition~\ref{wk}.
\end{proof}

Another variation on the eventual Helfgott delta is the supremum over
$\delta$ such that $|S^3|\geq |S|^{1+\delta}$ for all symmetric generating
sets $S$ containing 1 of $\slp$ for sufficiently large $p$. We will show
in Corollary \ref{inf} that our subsets with 
 $\delta=(\log(7)-1)/6\approx 0.3012$ occur in $\slp$ for infinitely many
$p$, giving an upper bound for this variation of the eventual Helfgott
delta.  
\section{The exceptional subgroups}

The remaining subgroups to be considered are the exceptional subgroups
$2\udot\Alt_4, 2\udot\Sym_4$ and $2\udot\Alt_5$, of orders $24,\,48$ and 120 
respectively.
We deal with each case in turn.

\begin{prop} \label{a4}
Let $H \cong 2 \udot \Alt_4$ be a subgroup of $\slp$ for some $p$, and 
let $S$ be an $H$-plus-two subset or $H$ plus a coset core.
If $\langle S\rangle = \slp$ then
$L=x^{-1}Hx\cap H$ has index at least
3 in $H$ and $|S^3|\geq 96$, so that
$|S|^3\geq |S|^{1+\delta}$ for $\delta=\log(3)/5\approx 0.3169$.
\end{prop}
\begin{proof}
Note that $H$ has no subgroups of index 2.
Thus Lemma~\ref{lem:c_estimate},  with $|H|=24$
and $[H:L]\geq 3$, yields
$|S^3|\geq 
96$ and $|S|\leq 
32$.
\end{proof}

We now move to $H=2\udot\Alt_5$, because it turns out that 
$2\udot\Sym_4$ will produce
the lowest values of $\delta$.

\begin{prop} \label{a5}
If $\slp$ has a subgroup $H$ isomorphic to $2\udot\Alt_5$ then for any
$H$-plus-two subset or $H$ plus coset core  $S$ with 
$\langle H,x\rangle = \slp$ we can bound $|S|^3\geq |S|^{1+\delta}$ for 
$\delta=\log(5)/\log(144)\approx 0.3238$.
\end{prop}
\begin{proof}
The group $2 \udot \Alt_5$ has no proper subgroups of index less than $5$.
Thus Lemma~\ref{lem:c_estimate} implies that $|S^3|\geq 5|H|+|H|=720$ and 
$|S|\leq 120+24=144$. 
\end{proof}

We now come to 
the best possible value of $\delta$ over the two types of
subset considered and we conclude, perhaps surprisingly, 
that subgroup-plus-two subsets cannot obtain this value of $\delta$. 
Recall the types of matrices defined at the beginning of Section~\ref{sec:tri}, and that $2 \udot \Sym_4 \le \slp$ 
only when $p \equiv \pm 1 \bmod 8$, and is maximal for these $p$. 


\begin{thm} \label{s4}
Let $H \cong 2 \udot \Sym_4$ be a subgroup of $\slp$ for some $p$, and let 
$S$ be an $H$-plus-two subset or $H$ plus coset core 
with $\langle S\rangle=\slp$.
Then $|S^3|\geq 224$ and $|S|\leq 64$, giving
$|S^3|\geq |S|^{1+\delta}$ for $\delta=(\log(7)-1)/6
\approx 0.3012$.
Furthermore, $|S^3|=|S|^{1+\delta}$ if and only if 
$L=x^{-1}Hx\cap H$ has index 3 in $H$ and $S=H\cup xL$ with $x^2\in H$.  
\end{thm}
\begin{proof}
The group $2 \udot \Sym_4$ has a unique subgroup of index $2$,
so we can apply Lemma~\ref{unin} to conclude
that if $L$ has index $2$ then $L$ is normalised by $\lgl H, x \rgl = \slp$ which is a
contradiction. 

If $[H:L]\geq 4$ then Lemma~\ref{lem:c_estimate} gives 
$|S^3|\geq 
240$ and $|S|\leq 
60$, so we assume from now on that $[H:L]=3$.
Moreover we can assume
without loss of generality that $x^2\in H$ when finding the smallest value of
$|S^3|$. As for $|S|$, 
if $x^2\notin H$ then
$S = H \cup \{x^{\pm 1}\}$ and so $|S|=50$, whereas if $x^2\in H$
then we can take $S$ to be
the subgroup plus coset core of size 64.
 
Thus we will assume from now on that $x^2\in H$ and $[H:L]=3$ so 
$S^3 = H \cup HxH\cup x^{-1}Hx$. 
Therefore we will obtain the given value for $|S^3|$
on showing that
$HxH \cap x^{-1}Hx = \emptyset$. 
To do so, we will work in the characteristic zero representation of $2 \udot \Sym_4$ given by
 $\overline{H}=\langle a,b\rangle$ where
\[a= \frac{\sqrt{2}}{2}\diag\left[(1+i), (1-i)\right], 
b=\frac{\sqrt{2}}{2}\sma{rc}1&1\\-1&
1\fma\]
so that $a$ and $b$ are of order 8. 
Our assertions in the remainder of this proof about $\overline{H}$ can 
easily be verified in \Magman, by defining $\overline{H}$ as the group generated by 
$a$ and $b$ over $\Q(\sqrt{2}, i)$. 

There is a unique faithful $2$-dimensional character of $H$, up to automorphisms.
Thus if $p \equiv 1 \bmod{8}$ then $H$ is the $p$-modular 
reduction of $\overline{H}$, whilst if $p \equiv -1 \bmod{8}$ then 
$H$ is a $\mathrm{GL}(2, p^2)$-conjugate of a $p$-modular 
reduction of $\overline{H}$. Let $\F$ be $\F_p$ when $p \equiv 1 \bmod 8$ 
and $\F_{p^2}$ otherwise, so that the $p$-modular reduction of $\overline{H}$ lies
in $\F$. 

We now proceed to work purely over $\mathbb{Q}(\sqrt{2}, i)$ but all
algebraic consequences will be true over $\F$ too: henceforth we 
identify $\overline{H}$ with $H$.
The group $L$ is a 
Sylow 2-subgroup of $H$, so it is straightforward to check that
without loss of generality we may define $c:= \frac{\sqrt{2}}{2}\antidiag[(-1+i), (1+i)]$
and set $L = \lgl a, c \rgl$.


As $x^{-1}Lx=L$ and there are only 2 elements of order 8  and trace $\mathrm{tr}(a)$ in $L$, namely
$a^{\pm 1}$,
we deduce that $x^{-1}ax= a^{\pm 1}$. 
An easy calculation tells us that if
$x^{-1}ax=a$ then $x= \diag[u, u^{-1}]$ for some $u$, 
whereas $x^{-1}ax=a^{-1}$ means that $x = \antidiag[v, -v^{-1}]$. 
Now as $x\neq\pm I$ but $x^2\in L$, the order of $x$ is 4, 8 or
16. Therefore $u^{16}=1$ in the first case, whereas a direct calculation in the second
case shows that $x$ has order 4 for any invertible $v$.

Let us start by considering the second case. Since $[H:L] = 3$, we 
define $z = \sqrt{2}i/2$ and fix right (and left) coset representatives $I$, 
$$d=\sma{rc}-z&z\\z&z\fma, \mbox{ and } e=\sma{ll}-z&zi\\-zi&z\fma.$$

If $HxH$ intersects $x^{-1}Hx$ nontrivially then $l_1sxtl_2=x^{-1}hx$ for some $h\in H$,
$l_1,l_2\in L$ and $s,t \in \{I,d,e\}$. As
$x$ normalises $L$, this is equivalent to saying that $sxt$ is in 
$x^{-1}Hx$. If $s$ or $t$ is $I$ then $sxt = x^{-1}hx$ implies that $x\in H$, so we
must check to see if any of $dxd,exe,dxe$ and $exd$ are in 
$x^{-1}Hx$,
though the last check is unnecessary because $exd \in x^{-1}Hx$ if and only if
its inverse $-dxe$ is (as $|d| = |e| = |x| = 4$), so 
if and only if $dxe$ is.

Now $dxd$ is easily confirmed to be of the form
\[-\frac{1}{2}\sma{lr}(v^{-1}-v)&-(v+v^{-1})\\(v+v^{-1})&(v-v^{-1})\fma\]
but let us consider the form of the order 4 elements in $x^{-1}Hx$. As
$x  = \antidiag[v, -v^{-1}]$, 
when an arbitrary element of 
$\mathrm{SL}(2,\mathbb F)$ is conjugated by $x$ the diagonal entries are swapped.
Moreover, a diagonal matrix remains diagonal under conjugation by $x$.
Now $dxd$ cannot be in $L$ as this would imply $x\in H$, so we need
to see if $dxd$ can be equal to $x^{-1} y x$ where $y$ is 
one of the eight elements of $H  \setminus L$ of order $4$. 
The sum of the antidiagonal entries
of $dxd$ is zero but standard calculations reveal that this only
happens for $x^{-1} y x$ if $v^8=1$. However, setting $v^8= 1$ yields that $x$ lies in $L$, a contradiction. 

Similarly
\[exe=
\frac{1}{2}\sma{rr}-i(v+v^{-1})&(v-v^{-1})\\(v-v^{-1})&i(v+v^{-1})\fma\]
and this time the off-diagonal entries are equal. 
Forcing this to occur for $x^{-1}yx$ implies that $v^8=1$.

We do not know a priori the trace of $dxe$.
Thus instead of checking whether $dxe$ can be in $x^{-1}Hx$, 
we will calculate whether $y:= xdxex^{-1}$ can lie in $H$. Now, 
\[y= \sma{rr} z^2(iv^{-1}-v)&-v^2z^2(iv-v^{-1})\\
          -v^{-2}z^2(v+iv^{-1})&-z^2(iv+v^{-1})\fma.\]
We first note that no
entry of $y$ can be zero because $z,v\neq 0$ and $v^8\neq 1$: 
this leaves 32 possible elements of $H$. 
Now, the ratio $y_{1,2}/y_{2,1} = -iv^4$, and looking through these 
elements of $H$, this must lie in $\{\pm1, \pm i\}$. 
If $iv^4 = \pm i$ then $v^8=1$, a contradiction as before. If however
$-iv^4=\pm 1$ then $v$ is a primitive $16$th root of unity. We set a first possible $v$ to be the square root of $\sqrt{2}(1+i)/2$, and check over $\Q(\sqrt{2}, i, v)$ that each
odd power of $v$ yields an $x$ such that $x^{-1}Hx \cap HxH = \emptyset$. 

Now we return to the case where $x = \diag[u, u^{-1}]$ 
for $u^{16}=1$. 
If $u^8 = 1$ then $x \in H$, so $x$ has
order 16, and as in the previous paragraph 
we can define $u$ to be a square root of $\sqrt{2}(1+i)/2$, and 
check over $\Q(\sqrt{2}, i, u)$  that each odd power of 
$u$ yields an $x$ such that $HxH \cap x^{-1}Hx = \emptyset$. 
\end{proof} 

We must also show that these best possible sets
do actually occur.

\begin{co} \label{inf}
Let $p$ be a prime with $p \equiv 1 \bmod 16$. Then 
$\slp$ contains a subgroup plus coset
core  $S$ of size 64 with $|S^3|=224$.
\end{co}
\begin{proof}
For such $p$ there are square roots of $-1$ and $2$ in
$\F_p$, and the characteristic zero representation of $2\udot\Sym_4$
given in Theorem~\ref{s4} embeds in $\slp$ and is maximal. 
Moreover, there exist elements $v\in \F_p^*$ of order 16. Thus set
$x = \antidiag[v, -v^{-1}] \notin H$, of order 4. Now $x^2=-I\in H$ and $\langle H,x\rangle=\slp$, and 
as the conjugate $x^{-1}mx$ of an arbitrary matrix
$m=\sma{cc}a&b\\c&d\fma$ is equal to $\sma{cc}d&-cv^2\\-bv^{-2}&a\fma$,
we see that $x^{-1}Lx=L$ so that $[H:L]\leq 3$. But this index cannot
be 1 or 2 by Proposition~\ref{ind12} so we can now apply Theorem~\ref{s4}.
\end{proof}  

We can now give our main result which follows immediately from this
and the two previous sections, given that all proper subgroups of
$\slp$ have now been covered.
\begin{co} \label{best}
Let $S$ be a subgroup-plus-two subset or subgroup plus coset core
 of $\slp$ with $\lgl S \rgl = \slp$. Then $|S^3|\geq |S|^{1+\delta}$ for
$\delta=(\log(7)-1)/6$. 
Moreover this value is obtained if and only if $H=2\udot\Sym_4$ with $x^2\in H$,
$[H : x^{-1}Hx\cap H] = 3$ and $S = H\cup (xH \cap Hx)$. In particular, 
subgroup-plus-two subsets do not attain the smallest 
possible value of $\delta$.
\end{co}

Recall that $(\log(7)-1)/6 \approx 0.30122$. 

\section{Further evidence}

We have proved that over all subgroup-plus-two subsets and subgroup plus coset
cores, those giving rise to the smallest value of $\delta$
are exactly the ones in Corollary~\ref{best}. But might they give the
best possible value over all symmetric generating
subsets $S$ containing 1 and with $S^3\neq\slp$,
thus providing us with the correct value of the Helfgott delta?
Clearly there are vastly many more subsets in this general form compared with
the restricted nature of the subgroup-plus-two subsets and subgroup plus coset
cores. Nevertheless it is our contention that the correct
value is much nearer $0.3012$ than the known lower bound
$1/3024\approx 0.0003$ in \cite{kwl}, and indeed these subsets might be
best possible. In order to provide further evidence for this, we show 
that these subsets are ``local minima'' in a very general sense.

To define this concept, first suppose that $S=H\cup xL$
is as in Corollary~\ref{best} and recall Proposition~\ref{More} which
states that if 
$T=S\cup\{y^{\pm 1}\} \neq S$ then $|T^3| > |S^3|$. 
We show that in fact $|T^3|$ is so much bigger than $|S^3|$ that the value of $\delta$ increases. In this section, for a subset $S$ of $\slp$, we write $\Delta(S)$ to denote
$\log(|S^3|)/\log(|S|)$ (this is one more than the value of $\delta$ for $S$). 
\begin{thm} \label{add}
Let $S=H\cup (xH \cap Hx)$ be as as in Corollary~\ref{best}, 
and $T=S\cup\{y^{\pm 1}\}$
for $y\notin S$. Then $\Delta(T) > \Delta(S)$. 
\end{thm}
\begin{proof}
For this $S$, we know that $S^3=H\cup HxH\cup x^{-1}Hx$, and that
$|HyH|\geq 3|H|=144$.  
So if
$HyH\neq HxH$ then 
the set $H\cup HxH\cup HyH$, of size at least 336, is a subset of $|T^3|$,
which means that $\Delta(T)$ 
 is much bigger than
$\Delta(S)$. 
If $HyH = HxH$ then $HyH = Hy^{-1}H$, 
so by Theorem~\ref{wlog}
there is $z=hy$ with $z^2\in H$ such that $H\cup HzH\cup z^{-1}Hz \subseteq 
T^3$. Thus $z = h_1 x h_2$ for some $h_1, h_2 \in H$, and so $z^{-1}Hz \cap H = h_2^{-1}Lh_2$. 
Hence, the conditions of Theorem~\ref{s4} are satisfied
we conclude that $z^{-1}Hz$ is disjoint from $HzH$.

If $x^{-1}Hx=z^{-1}Hz$ then $xz^{-1}$ is in the normaliser of the self-normalising
subgroup $H$ so $z\in Hx$. But $xHx^{-1}=x^{-1}Hx$ and the same holds for
$z$, so repeating this argument gives $z\in xH$ and hence $z$ was in $S$
anyway, a contradiction. 

Thus we can assume that $x^{-1}Hx \neq z^{-1}Hz$, 
and that both of these subgroups are disjoint from $HxH=HzH$ and contained in
$T^3$. Now $z^{-1}Hz \cap H$ is conjugate to $L$, so 
$|z^{-1}Hz\cap H| = 16$. 
%
Since $z^{-1}Hz\neq x^{-1}Hx$, the group $z^{-1}Hz\cap x^{-1}Hx$ 
has index at least 2 in $x^{-1}Hx$, thus $z^{-1}Hz$ has at most 24
elements in $x^{-1}Hx$. Now, any two Sylow $2$-subgroups of $H$ intersect in a 
group of order $8$, so $z^{-1}Hz \cap (H \cap x^{-1}Hx)$ has order at least 8. 
Hence, at least 8 elements of $z^{-1}Hz$ have been double
counted when looking at which ones lie in $H$ and in $x^{-1}Hx$, so at
most 32 elements of $z^{-1}Hz$ are in $x^{-1}Hx\cup H$. This leaves at
least 16 extra elements, making $|T^3|\geq 240$ and $|T|=66$, so
$\Delta(T) > 1.3081$.
\end{proof}

Another reasonable definition of local minimum is that the $\delta$
increases under the removal of any element and its inverse.

\begin{thm} \label{sub}
Let  $S=H\cup xL$ be as in Theorem~\ref{best}, and
let $T=S\setminus\{z^{\pm 1}\}$ for some $z\in S$. Then
$\Delta(T) > \Delta(S)$. 
\end{thm}
\begin{proof}
First assume that $z\in H$ and that $z\neq z^{-1}$ (so that we have removed
two distinct points). We will write $h$ for $z$ and set $H_0=H-\{h^{\pm 1}\}$.
We will show that $T^3=S^3$, which we know to be $H\cup HxH\cup x^{-1}Hx$.

A very old and straightforward result states that if $A,B$ are subsets of a
finite group $G$ with $|A|+|B|>|G|$ then $AB=G$. Thus $H=H_0^2\subseteq T^3$.
In order to show that $HxH\subseteq T^3$, 
it suffices to show that $T^3$ contains $H_0xh^{\pm 1}, h^{\pm 1}xH_0$ and $h^{\pm 1}xh^{\pm 1}$ (for all choices
of signs). We choose any $l\in L$ 
such  that
$l^{-1}h^{\pm 1}$ is not equal to $h$ or $h^{-1}$ and thus is in $H_0$. Then
$H_0xh^{\pm 1}=H_0\cdot xl \cdot l^{-1}h^{\pm 1} \subseteq H_0xLH_0$ and
so certainly is in $T^3$.

This also applies to $h^{\pm 1}xH_0$ so we are left with $x^{-1}Hx$. We 
clearly already have $x^{-1}H_0x\subseteq T^3$ so just need $x^{-1}h^{\pm 1}x$.
If $h\in L$ then $x^{-1}h^{\pm 1}x\in L \subseteq T^3$, so assume that 
$h\in H-L$. Then we are done if we can find $m\in L$ such that
$m^{-1}hm\neq h^{\pm 1}$, because 
$x^{-1}h^{\pm 1}x=x^{-1}m\cdot m^{-1}h^{\pm 1}m\cdot m^{-1}x\in 
xL\cdot H_0\cdot xL$. It is easy to check that in $\Sym_{4}$, any element 
$\overline{h}$ outside a Sylow 2-subgroup $\overline{L}$ satisfies 
$|\mathrm{C}_{\Sym_{4}}(\overline{h}) \cap \overline{L}| \le 2$, 
so the number of elements of $L$ that either centralise or invert 
$h \in H \setminus L$ is at most $8$, and such an $m$ exists. 

We next consider when $H_0$ is formed by removing just $-I$ from
$H$. The same arguments as above apply to
show that $H$ and $HxH$ are in $T^3$, and when we compare $x^{-1}H_0x$ to
$x^{-1}Hx$ we see we are only missing $-I$ which is already in $H$ and so 
in $T^3$.

Finally, consider what happens if we remove an element $lx$ and
its inverse from $Lx=xL$ to form $T$. On taking $m\in L$ such that
$mx\neq (lx)^{\pm 1}$ and thus is in $T$, we obtain $HxH=Hm^{-1}\cdot mx\cdot
H=HmxH\subseteq T^3$ and $x^{-1}Hx=(mx)^{-1}Hmx\subseteq T^3$, with
$H\subseteq T^3$ already.
\end{proof} 

We now obtain our final result on local minima, where this time we 
allow ourselves
to remove an element and its inverse from $S$, then replace it by an
arbitrary element and inverse from outside $S$ to form $T$.
\begin{co}
Let $S$ be as in Theorem~\ref{best}, let $1 \neq s \in S$ and
$y\in \slp \setminus S$, and let $T=(S \setminus\{s^{\pm 1}\})\cup \{y^{\pm 1}\}$. 
Then $\Delta(T) > \Delta(S)$. 
\end{co}
\begin{proof}
By Theorem~\ref{sub}, if we set 
$Z=S \setminus \{s^{\pm 1}\}$ then $Z^3=S^3$.
As $|T|=|S|$ or $|S|+1$ (the latter occurring only if we remove $-I$), we
will be done on showing that $|T^3|\geq |S^3|+14$ by finding elements that are 
not in $S^3$ but which can be made out of
$Z$ and $y^{\pm 1}$. On examining the proof of Theorem~\ref{add}, 
we note that elements 
in $(S\cup\{y^{\pm 1}\})^3 \setminus S^3$ came from $HyH$ or $Hy^{-1}H$
or $y^{-1}Hy$. Thus if $s \not\in H$ then these will also
be in $T^3$.

We now suppose that $s\in H$ and let $H_0=H\setminus\{s^{\pm 1}\}=Z\cap H$. First
say that $HyH$ (or $Hy^{-1}H$ by changing $y$ to $y^{-1}$) provides new
elements for $(S\cup\{y^{\pm 1}\})^3$. As $|HyH|$ is at least $3|H|$, the double coset
$HyH$ contains at least 3 left cosets of $H$. This implies that
$|H_0yH|\geq |H|$ because although we could be missing the two left cosets
$syH$ and $s^{-1}yH$ when we drop from $HyH$ to $H_0yH$, there will still be
at least one left over. This in turn means that $|H_0yH_0|\geq |H|-2$ and so
there are at least 46 extra elements in $T^3$.

Finally if our extra elements came from $y^{-1}Hy$ then we still
have all but two in $y^{-1}H_0y$, and in the proof of Theorem~\ref{add} we
showed that the former set introduces at least 16 extra elements, so the
latter provides at least 14.
\end{proof}

It might well be so that our subsets $S$ remain best possible
under the removal or addition of two (or more) elements and their inverses,
although we have not examined this
owing to the lengthier number of cases to consider.

\section{Computer calculations}

The main computer calculation that we did was an exhaustive search through 
$\mathrm{SL}(2, 5)$ looking for the sets $S$ of minimal tripling. There are $2^{120}$ 
potential such subsets, so we implemented a backtrack search as follows. For 
convenience we split the search in two, one for sets $S$ containing $-I$, and 
one for the remaining sets $S$. The set $S$ was initialised to $\{I\}$ 
or $\{\pm I\}$ and was then grown by adding elements $x, x^{-1}$ 
at each branch point.  For the first few levels of the search tree (up 
to depth around 3) we only chose $\{x, x^{-1}\}$ up to conjugacy under 
the subgroup of $\mathrm{SL}(2, 5)$ that conjugated each element of $S$ to itself or its inverse. 
After this we chose all possible $x$, as in $\mathrm{SL}(2, 5)$ the stabiliser of a triple of 
elements and their inverses
is likely to be just $\langle -I \rangle$. The search stored the corresponding $\delta$ whenever $S$  generated $\mathrm{SL}(2, 5)$, and backtracked when $S^3$ became equal to $\mathrm{SL}(2, 5)$. The following result has since been confirmed independently by Chris Jefferson, who also showed that all sets $S$ attaining the bound are conjugate under $\mathrm{GL}(2, 5)$. 

\begin{thm}
Let $S$ be a subset of $\mathrm{SL}(2, 5)$ such that $1 \in S$, $S = S^{-1}$ and $\langle S \rangle = \mathrm{SL}(2, 5)$. Then $|S^3| \geq |S|^{1.3925}$, and the set $S$ closest to this bound has size $30$ with $|S^3| = 114$.
\end{thm}

One such optimal $S$ is the following elements and their inverses
$$\begin{array}{c}
\left( \begin{array}{cc}
2 & 0\\
0 & 3
\end{array} \right),
\left( \begin{array}{cc}
3 & 0 \\
1 & 2 \\
\end{array} \right),
\left( \begin{array}{cc}
0 & 3 \\
3 & 2 \\
\end{array} \right),
%
%
\left( \begin{array}{cc}
4& 3\\
 2 & 3 
\end{array} \right),
%
%
\left\langle \left( \begin{array}{cc}
1 & 1\\
4 & 0 \\
\end{array} \right) \right\rangle, \\
\left\langle \left( \begin{array}{cc}
1 & 4 \\
 1 & 0 \\
\end{array} \right) \right\rangle,
\left( \begin{array}{cc}
3 & 3 \\
3 & 0 \\
\end{array} \right),
%
%
\left( \begin{array}{cc}
2 & 3 \\
2 & 1 \\
\end{array} \right),
\left\langle \left( \begin{array}{cc}
1 & 1 \\
1 & 2 \end{array} \right) \right\rangle. \end{array}
$$


For larger $p$, we decided that 
there was no point examining extremely small subsets of $\slp$ systematically, since it is an easy exercise to see that any $S$ of order $5$ (say) would satisfy $|S^3|> 10 > 5^{1.4}$ (say), and hence never be a set of minimal $\delta$. Thus the sets $S$ need to be reasonably large, and the combinatorial explosion in the number of possible sets would seem to preclude a systematic search.

Similarly, one would not expect a random subset of $\slp$ to have a low value of $\delta$, so extensive random sampling does not seem likely to be useful.

The final obvious trick for computational exploration would be to ``evolve'' sets $S$ by adding elements whenever $S^3$ doesn't grow (or possibly doesn't grow by too much), and otherwise interchanging elements in $S$ for elements outside $S$ when this reduces or stabilises the size of the triple product. However, this would need to be very carefully designed to avoid the search getting stuck at local minima for $\delta$ that are not global minima. 

We finish with a brief word on subsets with small triple products in
other infinite families of finite simple (or almost simple) groups. First
we mention $\mathrm{PSL}(2, p)$: Helfgott's result is sometimes stated for 
this case
but in general one works in $\slp$ for added convenience. However it is
certainly straightforward to go from $\slp$ to $\pslp$. Suppose that
we know a value of $\delta$ where $|A^3|\geq |A|^{1+\delta}$ for any
symmetric generating subset $A$ containing 1 and with $A^3\neq\slp$.
Now suppose there exists $B\subseteq\pslp$ which is symmetric, generates,
contains 1 but with $B^3\neq\pslp$. Then the pullback $A=\pi^{-1}(B)$ is also
symmetric, generates $\slp$, contains 1 and satisfies $A^3\neq \slp$.
Moreover $|A|=2|B|$ and $|A^3|=2|B^3|$ because 
$(\pi^{-1}(B))^3=\pi^{-1}(B^3)$ for surjections $\pi$. Thus 
$|B^3|\geq |A|^{1+\delta}/2\geq 2^\delta |B|^{1+\delta}\geq |B|^{1+\delta}$,
meaning that the Helfgott delta in $\pslp$ is at least that for $\slp$.
For instance our subset in Theorem~\ref{s4} gives rise to a
subset $B$ of $\pslp$ of size 32 with $|B^3|=112$, thus giving an upper
bound of $0.3614$ for the Helfgott delta in $\pslp$.

In addition to the Helfgott delta,
the general results of \cite{ps} and \cite{bgt} show that
for any family of finite simple groups of Lie type of bounded rank, 
there exists some some delta holding for all groups
in the family. 
However this breaks down without bounded rank, for instance in \cite[Section 14]{ps}
counterexamples are given for $\Sym_n$ and for 
$\mathrm{SL}(n,p)$ where $n$ varies. 
Interestingly, the first counterexample is a sequence of subgroup-plus-two
subsets, and the other is what we would call here
subgroup-plus-four subsets.

\section{Acknowledgements}
We thank the anonymous referee for their many helpful suggestions, which have greatly improved the paper. Colva Roney-Dougal acknowledges the support of EPSRC grant EP/I03582X/1.

\end{document}